\documentclass[11pt]{amsart}
\usepackage{amsmath,amssymb,enumerate,amsfonts,amsthm,graphicx,color}
\usepackage{eucal,bbm,amscd}
\usepackage{mathrsfs}
\usepackage[all]{xy}
\usepackage{tikz}
\usetikzlibrary{arrows,snakes,backgrounds}
\usetikzlibrary{topaths}
\usepackage{tikz}  % To draw figures
\usetikzlibrary{arrows,snakes,backgrounds}
\usetikzlibrary{topaths} % LATEX and plain TEX
\usetikzlibrary[topaths] % ConTEXt

\newtheorem{thm}{Theorem}[section]
\newtheorem{cor}[thm]{Corollary}

\newtheorem{lem}[thm]{Lemma}

\theoremstyle{definition}
\newtheorem{defn}[thm]{Definition}
\theoremstyle{definition}
\newtheorem{rem}[thm]{Remark}
\theoremstyle{definition}
\newtheorem{exm}[thm]{Example}
\theoremstyle{definition}
 
\newcommand{\reg}{\mbox{\rm reg}}

  \newtheorem{thevarthm}[thm]{\varthmname}

\newenvironment{varthm*}[1]{\trivlist\item[]{\bf #1.}\it}{\endtrivlist}

\usepackage[naturalnames=true]{hyperref}

\title
[Containment problem for the quasi star configurations  $\dots$]{Containment problem for the quasi star configurations of points in $\mathbb{P}^2$}
\author{Hassan Haghighi, Mohammad Mosakhani }
\address{Hassan Haghighi, Mohammad Mosakhani, \ Faculty of Mathematics, K. N. Toosi
University of Technology, Tehran, Iran.}
  \email{haghighi@kntu.ac.ir, mosakhani@aut.ac.ir}

\keywords{Configuration of points, Symbolic power, Resurgence, Waldschmidt
constant, Quasi star configuration, Containment problem} \subjclass[2010]{Primary 13A15, 14N20; Secondary 13F20,
14N05 }
\begin{document}
\maketitle
 \vspace{-0.5cm}
\begin{abstract}
 In this paper,  the containment problem for the
defining ideal of a special type of zero dimensional subschemes of
$\mathbb{P}^2$, so called quasi star configurations, is
investigated. Some sharp bounds for the resurgence of these types of
ideals are given.  As an application of this result, for every real
number $0 < \varepsilon < \frac{1}{2}$, we construct an infinite
family of homogeneous radical ideals of points in
$\mathbb{K}[\mathbb{P}^2]$ such that their resurgences lie in the
interval $[2- \varepsilon ,2)$. Moreover, the Castelnuovo-Mumford
regularity of all ordinary powers of defining ideal of quasi star
configurations are  determined. In particular, it is shown
that, %the defining ideal of a quasi star configuration, and
 all of them have linear resolution.
\end{abstract}
$\vspace{5mm}$

\section{Introduction}
Let $\mathbb{K}$ be an  algebraically closed field and let
$R=\mathbb{K}[x_0, \dots , x_N]= \mathbb{K}[\mathbb{P}^N]$ be the
coordinate ring of the projective space $\mathbb{P}^N$. Let $I$ be a
nontrivial homogenous ideal of $R$. For each positive integer $m$,
two different kinds of powers of $I$ can be constructed. The first
one, and the most algebraic one, is the ordinary power $I^m$ of $I$,
generated by all products of $m$ elements of $I$. The second one, is
the symbolic power $I^{(m)}$ of $I$, which is defined as follows:
\begin{align*}
I^{(m)}= R \cap I^mR_U,
\end{align*}
where $U$ is the multiplicative closed set $R-\bigcup_{P \in {\rm Ass}(I)} P.$
The $m$th symbolic power of $I$, rather than
the algebraic nature, has a geometric nature.
For example,
if $I \subseteq \mathbb{K}[\mathbb{P}^N]$ is the radical ideal
of a finite set of points $ p_1, \dots, p_n \in \mathbb{P}^N$,
then $I^{(m)}$,
 which is called {\it fat points ideal},
geometrically is defined as the ideal of all homogeneous forms vanishing to order at
least $m$ at all points $p_i$, and algebraically is defined as
$I^{(m)}= \cap_i I(p_i)^m$, where $I(p_i)$ is the ideal of
polynomials vanishing at the point $p_i$.
From the above definitions, immediately follows that $I^m \subseteq I^{(m)}$.

In recent years, comparing the behavior of symbolic powers and
ordinary powers of $I$ has prompted many challengeable problems and
conjectures in algebraic geometry and commutative algebra (see
\cite{BCH, Ha_Hu, SS}). One of these problems, known as {\it
containment problem},  asks for what pairs of positive integers
$(m,r)$ one may has $I^{(m)} \subseteq I^r$. The containment problem
has an asymptotic interpretation. Indeed, instead of searching for
pair  of integers $(m,r)$ such that $I^{(m)} \subseteq I^r$, one can
approach to the problem asymptotically. Harbourne and Bocci
\cite{BH} introduced an asymptotic invariant, known as the {\it
resurgence}, as follows:
\begin{align*}
\rho(I)=\sup \{ \frac{m}{r} \mid I^{(m)} \nsubseteq I^r \},
\end{align*}
where measures the discrepancies between the symbolic powers of a
homogenous ideal and its ordinary powers. The resurgence of $I$
exists and from its definition immediately follows that
\begin{equation}\label{17}
I^{(m)} \subseteq
I^r  \text{ if } \frac{m}{r} > \rho(I).
\end{equation}
In dealing with the containment problem, Ein, Lazarsfeld, and Smith
in \cite{ELS} as well as, Hochster and Huneke in \cite{Ho-Hu}, showed, but
with different methods,  that  for
all pairs of positive integers $(m,r)$ such that $m \ge Nr$, the containment
$I^{(m)} \subseteq I^{r}$ holds. In particular this implies
 $1 \leq \rho(I) \leq N$.

 Nevertheless, computing this numerical invariant  of a homogeneous ideal $I$ is
 difficult task and there are only few ideals $I$ for which the exact value of $\rho(I)$ is known (see \cite{BH,DHNSZT,BCH, BH1}).

 In this paper,  we study the containment problem of the defining ideal
of a special kind of  configuration of points in $\mathbb{P}^2$, so
called quasi star configuration (see Definition \ref{def}), which we
denote it by $Z_d$. Our main result, provides upper and lower bounds
for
% $\reg(I(Z_d))=\alpha(I(Z_d))$ and then we
$\rho(I(Z_d))$ as follows:
\\\\
{\bf Main Theorem} (Theorem \ref{main th}){\bf .} Let $I$ be the
defining ideal of a quasi star configuration of points $Z_d$ in
$\mathbb{P}^2$. If $d \geq 10$ then $2-\frac{2}{\sqrt{d}+1} \leq
\rho(I) \leq 2- \frac{2}{d+1}$.
\\\\
This result, with more details, is proved in Example \ref{exm1} and
Theorem \ref{main th}. Furthermore, as a corollary to the above
theorem, we  obtain the following result.
\\
\\
{\bf Theorem} (Corollary \ref{cor2}){\bf .} Let $ 0< \varepsilon <
1/2$ be a real number. Then there exists a radical ideal
%of points
$I_{\varepsilon}$
in $\mathbb{K}[\mathbb{P}^2]$ such that
$\rho(I_{\varepsilon}) \in [2- \varepsilon ,2).$
\\

 In \cite[Corollary 1.1.1]{BH}, it is
shown that for any homogeneous ideal $I\subset
\mathbb{K}[\mathbb{P}^N]$ the containment $I^{(rN)} \subset I^r$,
where $r$ is a positive integer,  is optimal. Whenever $N=2$, we can
show this optimality can be achieved via quasi star configurations
(see Corollary \ref{cor}).
%As a special case, for a general homogeneous ideal $I$ in
%$\mathbb{K}[\mathbb{P}^2]$ we see $I^{(4)} \subseteq I^2$. When $I$
%is a radical ideal of points, this containment can be narrowed. With
%this motivation, Huneke asked the following question.
%\begin{que}
%Let $I$ be the radical ideal of a finite set of points in $\mathbb{P}^2$.
%Is it always true that $I^{(3)} \subseteq I^2 ?$
%\end{que}

%Now, it is known that there are some configurations of points in
% $\mathbb{P}^2$ with defining ideal $I$ such that the containment $I^{(3)} \subseteq I^2$
%may fail. The first example which gave a negative answer to the
%Huneke's question, announced in 2013 in \cite{DST1}. Then after
%this, immediately other counterexamples were  found (see \cite{SS,
%Se, 123, HS2}).
% But for a general homogeneous ideal $I$, the failure of the containment $I^{(5)} \subseteq I^3$, or
% more generally, the failure of $I^{(2r-1)} \subseteq I^r$, whenever $r>2$,
%have remained still open.
By (\ref{17}), a necessary condition for the failure of $I^{(2r-1)}
\subseteq I^r$, for $r \geq 2$, is that $\rho(I) \geq
\frac{2r-1}{r}$, but possibly this is not a sufficient condition for
the containment. One of our results, in this paper is to show that,
for every integer $r \geq 2$, there exists a radical ideal  $I$ such
that meets
 this necessary condition. More generally:
\\\\
{\bf Theorem}{ (Corollary \ref{cor1})}{\bf.} Let $r \geq 2$ be an
integer. Then there exists a configuration of points in
$\mathbb{P}^2$  such that its defining ideal $I$ gives the necessary
condition for the failure of the containment $I^{(2r-1)} \subseteq
I^r$, i.e., $\rho(I) \geq \frac{2r-1}{r}$.

\section{Preliminaries}
Among the numerical  invariants of an arbitrary homogeneous ideal
which have been developed to study the containment problem, the {\it
Waldschmidt constant}, plays as a decisive role. This constant is
defined  as:
\begin{align*}
\widehat{\alpha}(I)=\underset{m \rightarrow \infty}{\lim}
\frac{\alpha(I^{(m)})}{m} =\underset{m \geq 1}{\inf} \frac{\alpha(I^{(m)})}{m},
\end{align*}
where $\alpha(I)$ is the least degree of a nonzero polynomial in $I$
and is called the {\it initial degree} of $I$. It is shown that this
limit exists \cite[Lemma 2.3.1]{BH}. The containment $I^m \subseteq
I^{(m)}$  holds for any positive integer $m$, which implies
$\alpha(I^{(m)}) \leq \alpha(I^m)=m\alpha(I)$. Therefore
$\widehat{\alpha}(I) \leq \alpha(I)$ and consequently
$\alpha(I)/\widehat{\alpha}(I) \ge 1$. Moreover, see \cite[Section
2.1]{456}, we have
\begin{equation}\label{eq}
\frac{\alpha(I^{(m)})}{m+1} \leq \widehat{\alpha}(I) \leq
\frac{\alpha(I^{(m)})}{m}.
\end{equation}

In general, computing the resurgence of an arbitrary homogeneous
ideal $I$ is quite difficult. However, whenever $I$ is the defining
ideal of a zero dimensional subscheme in $\mathbb{P}^N$, Bocci and
Harbourne \cite[Theorem 1.2.1]{BH} used the another numerical
invariants of $I$, i.e., the  Castelnuovo-Mumford %regularity of $I$
regularity $\reg(I)$, its initial degree, and the Waldschmidt
constant to bound $\rho(I)$ in terms of these invariants as follows:
\begin{align*}
\frac{\alpha(I)}{\widehat{\alpha}(I)} \leq \rho(I) \leq
\frac{\reg(I)}{\widehat{\alpha}(I)}.
\end{align*}
Recall that, if $0 \rightarrow F_r \rightarrow \dots \rightarrow F_i
\rightarrow \dots \rightarrow F_0 \rightarrow I \rightarrow 0$ is a
minimal free resolution of $I$ over $R$, then $\reg(I)$ is defined
to be $\max \{ f_j -j \mid j \geq 0 \} $, where $f_j$ is the maximum
degree of the generators of the free module $F_j$. In particular, we
have
\begin{equation}\label{19}
\rho(I) = \frac{\alpha(I)}{\widehat{\alpha}(I)} \text{\ \ if \ }
\reg(I)=\alpha(I).
\end{equation}
But it is rare to happen
$\reg(I)=\alpha(I)$, and even if the equality holds, it is a
hard task to compute $\reg(I)$ and $\widehat{\alpha}(I)$.
Nevertheless, computing these two invariants may be easier than
computing $\rho(I)$.

\section{Linear free resolution of a quasi star configuration}
The authors of \cite{DST2}, in the process of classification of all
configurations of reduced points in $\mathbb{P}^2$ with the
Waldschmidt constant less than $9/4$,   introduced a special type of
configuration of points, which called it {\it quasi star
configuration}, and is constructed from what is known as {\it star
configuration} that was introduced in \cite{GHM}. In the following
we recall the definition of quasi star configuration.

\begin{defn} \label{def} \cite[Definition 2.3]{DST2}.
Let $S_2(2,d)=p_1 + \dots + p_{\binom{d}{2}}$ be a star
configuration of points in $\mathbb{P}^2$ which is obtained by pair
wise intersections of $d \geq 3$ general lines $L_1, \dots , L_d$
and let $T_d= \{ q_1, \dots , q_d\}$ be a set of $d$ distinct points
in $\mathbb{P}^2$, such that $T_d \cap S_2(2,d)=\varnothing$. We say
that $Z_d=T_d+S_2(2,d)$ is a quasi star configuration of points if
 for each $i=1, \dots ,d$, the point $q_i$,
  lie on the line $L_i$. Moreover, these points are not collinear.
  \end{defn}
In the above definition, the points of $T_d$ need not to lie on a line. But if all points of $T_d$ are collinear,
 then $Z_d$ would be the star configuration
$S_2(2,d+1)$.% In this paper, we strictly assumed that the extra
%points of $T_d$ are not collinear.

The quasi star configuration of points $Z_5$ is
depicted in the figure below.
\begin{center}
\begin{tikzpicture}
%Five lines
\draw [-] [line width=1.pt] (-.8,0)--(5.8,0);
\draw [-] [line width=1.pt] (2.2,2.44)--(.8,-2.74);
\draw [-] [line width=1.pt] (1.5,2.625)--(4.5,-2.925);
\draw [-] [line width=1.pt] (-1,.5)--(4.5,-2.25);
\draw [-] [line width=1.pt] (.5,-2.25)--(6,.4);
%Label of Five lines
\draw [black] (6.2,0) node{$L_1$};
\draw [black] (5.9,.7) node{$L_2$};
\draw [black] (2.5,2.2) node{$L_3$};
\draw [black] (1.3,2.25) node{$L_4$};
\draw [black] (-1,.8) node{$L_5$};
% Ten double points
\shade[ball color=black] (0,0) circle (2.2pt);
\shade[ball color=black] (5.2,0) circle (2.2pt);
\shade[ball color=black] (4,-2) circle (2.2pt);
\shade[ball color=black] (1,-2) circle (2.2pt);
\shade[ball color=black] (2,1.7) circle (2.2pt);
\shade[ball color=black] (1.54,0) circle (2.2pt);
\shade[ball color=black] (2.918,0) circle (2.2pt);
\shade[ball color=black] (2.55,-1.27) circle (2.2pt);
\shade[ball color=black] (1.35,-.675) circle (2.2pt);
\shade[ball color=black] (3.378,-.855) circle (2.2pt);
%Five simple points
\shade[ball color=black] (4.2,0) circle (2.2pt);
\shade[ball color=black] (1.76,.8) circle (2.2pt);
\shade[ball color=black] (2.48,.8) circle (2.2pt);
\shade[ball color=black] (.64,-.31) circle (2.2pt);
\shade[ball color=black] (4.2,-.46) circle (2.2pt);
%%%%%%
%\draw [thick](3,-3.5) node {Figure 1. $W=\sum_{i=1}^{5} q_i + 2S_2(2,5)$
\draw [thick](3,-3.5) node {Figure 1
};
\end{tikzpicture}
\end{center}
Our next goal is to show that the defining ideal of a quasi star
configuration $Z_d$ has a linear minimal free resolution. To do
this, the following lemma is needed. Recall that whenever $J$ is a
saturated ideal in $R=\mathbb{K}[\mathbb{P}^N]$, the multiplicity of
$R/J$, denoted by $e(R/J)$, is equal to the degree of the closed
subscheme  associated to $J$.
\begin{lem}\label{lem}
Let $d$ be a positive integer and let $I$ be the defining  ideal of
$t=d(d+1)/2$ reduced points $\{p_1, \dots, p_t\}$ in
$\mathbb{P}^2$. Let $J$ be a saturated homogeneous ideal of
$R=\mathbb{K}[\mathbb{P}^2]$ such that $J \subseteq I$ and let $J$
has the following minimal free resolution
\begin{align*}
0\rightarrow R^{\beta_2}(-d-1) \rightarrow R^{\beta_1}(-d)
\rightarrow J \rightarrow 0.
\end{align*}
 Then $J=I$.
\end{lem}
\begin{proof}
Let $X(I)$ and $X(J)$ be the subschemes  correspond to the ideals
$I$ and $J$, respectively.  Since $J \subseteq I$, the support of
$X(I)$, i.e., $\{ p_1, \dots , p_t \}$, should be contained in the
support of the scheme $ X(J)$. Moreover, the minimal free resolution
of $J$ implies that the projective dimension of $J$ is equal to two
and hence $X(J)$ is a zero-dimensional subscheme of $\mathbb{P}^2$.
Applying a theorem of Huneke and Miller (\cite[Theorem 1.2]{HM}) to
this  minimal free resolution, implies that %the multiplicity $e(R /
%J)$ of the ring $R/J$ is equal to
 $e(R / J)=d(d+1)/2=t$. On
the other hand, since the multiplicity of the coordinate ring of  a
finite set of reduced points is equal to the number of its points,
hence $e(R/I)=t$. But, since $I$ and $J$ are the ideals of points
such that $J \subseteq I$ and $e(R/I)=e(R/J)$, therefore $I=J$, as
required.
\end{proof}

Now we are ready to show  $I(Z_d)$ has a linear minimal free
resolution.
\begin{thm}\label{the3}
Let $I\subset R=\mathbb{K}[\mathbb{P}^2]$ be the ideal associated to
quasi star configuration of points $Z_d=T_d+S_2(2,d)$ in
$\mathbb{P}^2$. Then the resolution
\begin{align*}
0\rightarrow R^d(-d-1) \rightarrow R^{d+1}(-d) \rightarrow I \rightarrow 0
\end{align*}
is the minimal free resolution of $I$.
In particular, $\alpha(I)=\reg(I)=d$.
\end{thm}

\begin{proof}
Let the star configuration $S_2(2,d)=p_1 + \dots + p_{\binom{d}{2}}$
be obtained by pair wise intersections of $d \geq 3$ general lines
$L_1, \dots , L_d$ and let $T_d=\{ q_1, \dots , q_d \}$. Let
$L_i^\prime$, where $1 \leq i \leq d$, be the line through $q_i$
that does not pass through any of the other points $Z_d$. Let $A$ be
the matrix
\begin{equation*}
A=
\begin{bmatrix}
L_1 & 0 & \dots & 0 \\
0 & L_2 & \dots & 0 \\
\vdots & \vdots & \ddots & \vdots \\
0 & 0 & \dots & L_d \\
L^\prime_1 & L^\prime_2 & \dots & L^\prime_d
\end{bmatrix}
\in \mathscr{M}_{(d+1) \times (d)}(R),
\end{equation*}
where A has as the top $d \times d$ rows a diagonal matrix
with $L_1, L_2, \dots L_d$ along the diagonal and as the last row the vector
$[L^\prime_1, L_2^\prime , \dots , L_d^\prime]$.
Let $I(A)$ be the ideal generated by the maximal minors of $A$. Thus we have
\begin{align*}
I(A)=(L_1L_2 \dots L_d, L^\prime_1 L_2 \dots
L_d, L_1L^\prime_2 L_3 \dots L_d, \dots, L_1L_2 \dots L_{d-1}L^\prime_d ).
\end{align*}
Since the columns of matrix $A$ is the first syzygy of $I(A)$,
then this ideal has the following free resolution
\begin{displaymath}
\xymatrix{
0 \ar[r] &  R^d(-d-1) \ar[r]^{A} & R^{d+1}(-d) \ar[r] & \ar[r]  I(A) & 0 }.
\end{displaymath}
In the sequel, we show this resolution is minimal. First, we show
$I(A)$ is the ideal of points.
%Thus $R/J$ will be Cohen-Macaulay ring.
%In particular, $I(A)$ will be an unmixed ideal.
Indeed, since $I(A)$ is not a principal ideal and since
%these elements are product of $d$ linear forms and since
 there is not any linear form, for example $L$, such
that $L$ divides all  $d+1$ elements of $I(A)$, thus $I(A)$ is the defining ideal of
a zero dimensional subscheme
 in $\mathbb{P}^2$.

Since the coordinate ring of  a zero-dimensional subscheme in
$\mathbb{P}^2$ is always Cohen-Macaulay, so
$\dim(R/I(A))=\text{depth}(R/I(A))=1$. Thus the projective dimension
$R/I(A)$ is equal to two, which implies that the above free
resolution  for $I(A)$ is  minimal.

Finally, it is easy to see that $I(A) \subseteq I$.
Since two ideals $I$ and $I(A)$ satisfy in the conditions of Lemma
\ref{lem}, thus $I=I(A)$.
As a consequence, we have $\alpha(I)=\reg(I)=d$.
\end{proof}

Our next theorem is an extension of  \cite[Theorem 4.6]{789}. As a special case, it
reveals more characterizations of a  quasi star configuration of
points. To state it, we need to recall some
preliminaries.

Let $X=m_1p_1+ \dots + m_rp_r$ be a zero dimensional subscheme of
$\mathbb{P}^2$ and let $I=I(X)$ be the corresponding saturated ideal
of $X$ in $R=\mathbb{K}[\mathbb{P}^2]$. Recall that the Hilbert
function of $X$, denoted by $H(R/I,t)$, is a numerical invariant of
$X$, defined by $H(R/I,t):=\dim_{\mathbb{K}}(R/I)_t$, where
$(R/I)_t$ is the $t-$th graded component of $R/I$. Moreover, $X$ is
called has a generic Hilbert function if for all nonnegative
integers $t$, $H(R/I,t) =\min \{\binom{t+2}{2}, \deg X \}$, where
$\deg X=\sum_i \binom{m_i+1}{2}$. In particular, when $X$ is a
finite set of reduced points, then $\deg X=\vert X \vert$.
\begin{thm}\label{th reg}
Let $X$ be a finite set of reduced points in $\mathbb{P}^2$ and let
$I=I(X)$. Also let $\alpha(I)=\alpha$. Then the following conditions
are equivalent:
\begin{itemize}
\item[$(i)$] The ideal $I$ is minimally generated by
$\alpha +1$ generators of degree $\alpha$.
\item[$(ii)$] The scheme $X$ has the generic
Hilbert function and $\vert X \vert =\binom{\alpha+1}{2}$.
\item[$(iii)$] The ideal $I$ has a linear minimal free resolution as follows:
\begin{displaymath}
\xymatrix{ 0 \ar[r] & R^{\alpha}(-\alpha-1)
\ar[r] & R^{\alpha+1}(-\alpha)
 \ar[r] & \ar[r] I & 0. }
\end{displaymath}
\item[$(iv)$]For the ideal $I$, we have $\reg(I) =\alpha(I)$.
\item[$(v)$] For all $m \geq 1$, we have
$\reg(I^m)= m \reg(I) = m\alpha(I)$.
\item[$(vi)$]The ideal $I^2$ is minimally generated by
$\binom{\alpha+2}{2}$ generators of degree $2\alpha$.
\item[$(vii)$] The ideal $I^2$ has a linear minimal free resolution as follows:
\begin{equation*}
0 \rightarrow R^{\binom{\alpha}{2}}(-(2\alpha+2)) \rightarrow
R^{2\binom{\alpha+1}{2}}(-(2\alpha+1)) \rightarrow
R^{\binom{\alpha+2}{2}}(-2\alpha) \rightarrow I^2 \rightarrow 0.
\end{equation*}
\end{itemize}
\end{thm}
\begin{proof}
The implications $(i) \Rightarrow (ii)$ and $(ii) \Rightarrow (iii)$
follow from \cite[Theorem 4.6]{789} and its proof. Moreover,  $(iii)
\Rightarrow (iv)$ is immediate.

To show  $(iv)$  implies $(v)$, let $d(I)$ denote the maximum degree
of elements of a minimal set of generators of $I$. By definition of
the regularity, it is clear that $ d(I^m)= md(I) \leq \reg(I^m)$.
Moreover, since Krull dimension of $ R/I$ is one, by \cite[Theorem
1.1]{GGP}, $\reg(I^m) \leq m \reg(I)$. Since $\alpha(I)=\reg(I)$,
thus $d(I)=\reg(I)$. Therefore, these two latter inequalities imply
$\reg(I^m) = m \reg(I)$.

We show $(v)$ implies $(i)$. Since $\alpha(I)=\reg(I)$, so the
degree of all elements of a minimal set of generators of $I$ are
equal. Let $c$ be the number of a minimal set of generators of $I$
of degree $\alpha$. Now we have to show  $c=\alpha+1$. Since $I$ is
generated by $c$ generators of degree $\alpha$, hence
$H(R/I,\alpha)=\binom{\alpha+2}{2}-c$ and also
$H(R/I,t)=\binom{t+2}{2}$, for $t \leq \alpha-1$, in particular
$H(R/I,\alpha-1)=\binom{\alpha+1}{2}$. Since $X$ is a zero
dimensional subscheme in $\mathbb{P}^2$,  the Hilbert function of
$X$ would be constant for $t \geq \reg(I)-1=\alpha-1$, which is
equal to $\deg X$ (see \cite[Section 1.2]{BH}). Thus
$H(R/I,\alpha-1) = H(R/I,\alpha)$. In particular,
$\binom{\alpha+1}{2} = \binom{\alpha+2}{2}-c$, which yields
$c=\alpha+1$.

By \cite[Lemma 4.2]{789},  $(i)$ implies $ (vi)$. To show $(vi)$
implies $(i)$, let $\mathscr{G}$ be a set of minimal generators of
$I$
%Since $\alpha(I^2)=2\alpha(I)=2\alpha$, thus $\alpha(I)=\alpha$.
 and let $d=d(I)$ denote the maximum degree of
elements of $\mathscr{G}$. We claim that $d=\alpha$. On the
contrary, let $d
> \alpha$, so there exists two polynomials $F$ and $G$ in
$\mathscr{G}$ such that $\deg F=\alpha$ and $\deg G=d$. By
\cite[Theorem 2.3]{789}, the polynomial $FG$ is of degree
$(d+\alpha)$ and is an element of a minimal set of generators of
$I^2$. But $(d+\alpha) > 2\alpha$, which
 contradicts our assumption. Now, let $c$ be  the number of elements of degree $\alpha$ in $\mathscr{G}$.
 By \cite[Lemma 4.2]{789}, the ideal $I^2$ has
$\binom{c+1}{2}$ minimal generators of degree $2\alpha$. Therefore,
$\binom{c+1}{2}=\binom{\alpha+2}{2}$. In particular, $c=\alpha+1$.

By Theorem 2.3 of \cite{789}, $(iii)$ implies $(vii)$.
At last, the  minimality of  free resolution of $I^2$,
immediately gives the implication $(vii) \Rightarrow (vi)$.
\end{proof}
A closer look at to the proof of the implication $(iv) \Rightarrow
(v)$ in Theorem \ref{th reg},  shows that for each positive integer
$m$, the ideal $I^m$ has linear resolution. This has the worth to be
mentioned as a corollary.
\begin{cor}
Let $I$ be the ideal of a configuration of points in $\mathbb{P}^2$ such that $\reg(I)=\alpha(I)$.
Then all ordinary powers of $I$ has linear free resolution.
\end{cor}
\begin{rem}
In addition to the quasi star configurations,  two distinct classes
of configurations of points in $\mathbb{P}^2$ can be named, such that
their defining ideals meet the conditions of Theorem \ref{th reg}.
These are:
\begin{itemize}
\item[(a)] the star configuration $S_2(2,d+1)$, for which  $\reg(I(S_2(2,d+1)))=\alpha(I(S_2(2,d+1)))=d$, by
\cite[Lemma 2.4.2]{BH};
\item[(b)] for any integer $d\ge 1$, the configuration $X$ consists of
$n=\binom{d-1+2}{2}$ generic points,  in
$\mathbb{P}^2$, for which $\reg(I(X))=\alpha(I(X))=d$ (see \cite[Section
1.3]{BH}).
\end{itemize}
%infinitely many configurations of points $X$ in $\mathbb{P}^2$ such
%that their defining ideals of satisfy the conditions of the above
%theorem. To see this fact, let $t \geq 1$ be an integer and let $X$
%be a set of $\binom{t+2}{2}$ generic points in $\mathbb{P}^2$. Then
%$I(X)$ is generated in one degree, i.e.,
%$\reg(I(X))=\alpha(I(X))=t+1$ (see \cite[Section 1.3]{BH}). As
%another interesting example we can mention to the star configuration
%of points $S_2(2,d)$, where $d \geq 2$, in $\mathbb{P}^2$. For this
%configuration, by \cite[Lemma 2.4.2]{BH}, we have
%$\reg(I(S_2(2,d)))=\alpha(I(S_2(2,d)))=d-1$.
An interesting problem that may arise in this regard, is to classify all
configurations of points $Z$ in $\mathbb{P}^2$ such that $I(Z)$
satisfies in one of the equivalent  conditions of Theorem \ref{th
reg}.
\end{rem}
The following example shows that the above three mentioned classes
of configurations of points  may have different features that
distinguish them from each other.
\begin{exm}
Let $X$ be a configuration of 6 generic points in $\mathbb{P}^2$,
$Y=S_2(2,4)$, be a star configuration generated by 4 general lines
in $\mathbb{P}^2$ and finally let $W=Z_3=T_3+S_2(2,3)$ be the quasi
star configuration generated by 3 generic lines in $\mathbb{P}^2$.
Then by our assumptions,  the number of points in $Y$ and $W$ is the
same as the number of points in $X$, while their resurgences are
different.
\begin{itemize}
\item [(1)]  By
\cite[Corollary
1.3.1]{BH}, $\rho(X)=5/4$; %\frac{7+1}{\sqrt{36}}=\frac{4}{3}$;
\item[(2)] By \cite[Theorem
2.4.3]{BH}, $\rho(Y)=3/2;$ %\frac{2(8-2+1)}{8}=\frac{7}{4}$.
\item[(3)]  By Example \ref{exm1},
$\rho(W)=4/3$.
%$2-\frac{2}{\sqrt{8}+1} \le \rho(W) \le 2- \frac{2}{8+2}=1.8.$
\end{itemize}
\end{exm}

\section{Proof of the  main Theorem}

By Theorem \ref{the3}, for a quasi star configuration $Z_d$, we have
$\alpha(I(Z_d))=\reg(I(Z_d))=d$. Thus by $(\ref{19})$, computing
 $\rho(I(Z_d))$ depends on the computing of
$\widehat{\alpha}(I(Z_d))$. For the
initial case of $d$, i.e., $d=3$, we can compute the exact value of
$\rho(I(Z_d))$. In fact
%  as follows:
\begin{exm}\label{exm1}
Let $Z_3$ be the quasi star configuration of 6 points.  Theorem
\ref{the3}, implies $\alpha(I(Z_3))=\reg(I(Z_3))=3$ and by
\cite[Proposition 3.1]{DST2}, we have
$\widehat{\alpha}(I(Z_3))=9/4$. Then $(\ref{19})$ yields
$\rho(I(Z_3))=4/3$.
\end{exm}
However, finding the exact  value of
$\widehat{\alpha}(I(Z_d))$, whenever $d \geq 4$, %except whenever $Z_d$
%is a star configuration,
seems to be a hard
problem. Hence, it is reasonable to look for good bounds for $\widehat{\alpha}(I(Z_d))$.
In the sequel, our goal is to establish an upper bound
for the Waldschmidt constant %$\widehat{\alpha}(I(Z_d))$
of the defining ideal of a quasi star configuration $I(Z_d)$.

Recall that for a real number $a$, $\lceil a \rceil$
denotes the least integer greater than or equal to $a$, and $\lfloor a \rfloor$
denotes the greatest integer less than or equal to $a$.

\begin{thm}\label{th}
Let $d \geq 4$ be an integer and let $Z_d=T_d + S_2(2,d)$ be a quasi
star configuration of points in $\mathbb{P}^2$. Let $I=I(Z_d)$. Then
\begin{itemize}
\item[$(a)$]if $d \leq 9$ then $\widehat{\alpha}(I) \leq \frac{d+c_d}{2}$,
where $c_d=2,2, \frac{12}{5}, \frac{21}{8}, \frac{48}{17} \text{ and } 3$
for $d=4, \dots , 9$, respectively;
\item[$(b)$]if $d \geq 10$ then $\widehat{\alpha}(I) \leq \frac{d+\sqrt{d}}{2}$.
\end{itemize}
\end{thm}
\begin{proof}
$(a)$ Let $W_d=\{ p_1, \dots, p_d \}$ be a set of $4 \leq d \leq 9$
general points in $\mathbb{P}^2$. It is known that
$\alpha(I(W_d)^{(m)})=\lceil c_dm \rceil$ (see for example
\cite{H}). Since the points of $W_d$ are in general position then
$\alpha(I(T_d)^{(m)}) \leq \alpha(I(W_d)^{(m)})=\lceil c_dm \rceil$.
Let now $c_d=\frac{a}{b}$, where $a$ and $b$ are two positive
integers and let $D$ be the polynomial $D=(L_1 L_2 \dots L_d)^{bm}$.
It is clear that $D \in I(S_2(2,d)^{(2bm)})$ and $D \in
I(T_d)^{(bm)}$. Since $\alpha(I(T_d)^{(bm)}) \leq \lceil c_dbm
\rceil = \lceil \frac{ab}{b}m \rceil=am$, thus there exists a
polynomial of degree $am$, for example $F$, vanishes to order $bm$
along $T_d$. Thus $FD \in I^{(2bm)}$. Therefore
\begin{center}
$\widehat{\alpha}(I) \leq \frac{\alpha(I^{(2bm)})}{2bm} \leq
\frac{\text{deg}F + \text{deg}D}{2bm} = \frac{am + bmd}{2bm}=
\frac{\frac{a}{b}+d}{2} = \frac{d + c_d}{2}$.
\end{center}

$(b)$ To prove the assertion, we use the same method as in the proof
of $(a)$. Let $d \geq 10$ be an integer and let $W_d$ be a set of
$d$ points in $\mathbb{P}^2$. It is always true that
$\widehat{\alpha}(I(W_d)) \leq \sqrt{d}$ ( \cite[Proposition
3.4]{CH}). By inequality (\ref{eq}), for all $m \geq 1$, we have
$\frac{\alpha(I(T_d)^{(m)})}{m+1} \leq \widehat{\alpha}(I(T_d))$.
Therefore $\alpha(I(T_d)^{(m)}) \leq (m+1)\sqrt{d}$. It means that
there exists a polynomial of degree $\lfloor (m+1)\sqrt{d} \rfloor$,
for example $F^\prime$, vanishes to order $m$ along $T_d$. Let now
$D^\prime$ be the polynomial $D^\prime =(L_1 L_2  \dots L_d)^m$. It
is easy to see that $F^\prime D^\prime$ is an element of $I^{(2m)}$
and is of degree $md+\lfloor (m+1)\sqrt{d} \rfloor \leq
m(d+\sqrt{d})+\sqrt{d}$. If allowing $m$ tends to infinity, then
$\widehat{\alpha}(I) \leq (d+\sqrt{d})/2$. Hence the claim
stablishes.
\end{proof}
Now we can use Theorems \ref{the3} and \ref{th} to bound the
resurgence $\rho(I(Z_d))$ of defining ideal of quasi star configuration $Z_d$ as follows:
\begin{thm}\label{main th}
Let $I$ be the ideal associated to quasi
star configuration of points $Z_d$ in $\mathbb{P}^2$.
\begin{itemize}
\item[$(a)$]If $4 \leq d \leq 9$ then $2-\frac{2c_d}{d+c_d}
 \leq \rho(I) \leq 2- \frac{2}{d+1}$,
\item[$(b)$]If $d \geq 10$ then $2-\frac{2}{\sqrt{d}+1}
\leq \rho(I) \leq 2- \frac{2}{d+1}$,
\end{itemize}
where $c_d$ is the one which defined in Theorem \ref{th}.
\end{thm}
\begin{proof}
By Theorem \ref{the3}, we have $\reg(I)=\alpha(I)=d$. Thus by
(\ref{19}),
 $\rho(I)=\frac{\alpha(I)}{\widehat{\alpha}(I)}$.
Hence the statement depends on bounds for $\widehat{\alpha}(I)$. By
\cite[Proposition 3.1]{Ha_Hu}, for every finite set of points $W$ in
$\mathbb{P}^2$, we have $\frac{\alpha(I(W))+1}{2} \leq
\widehat{\alpha}(I(W))$. In particular, $\frac{d+1}{2} \leq
\widehat{\alpha}(I)$, which yields $\rho(I) =
\frac{\alpha(I)}{\widehat{\alpha}(I)} \leq  \frac{2d}{d+1}= 2-
\frac{2}{d+1}$. By Theorem \ref{th}, for $d \geq 10$, we have
$\widehat{\alpha}(I) \leq \frac{d+\sqrt{d}}{2}$, which yields
$\rho(I)= \frac{\alpha(I)}{\widehat{\alpha}(I)} \geq
\frac{2d}{d+\sqrt{d}} = 2-\frac{2}{\sqrt{d}+1}$. For $4 \leq d \leq
9$ with a similar argument, we can deduce $\rho(I) \geq
2-\frac{2c_d}{d+c_d}$, as required.
\end{proof}

%\begin{rem}
%By item (b) of Theorem \ref{main th}, as well as item (b) of Theorem
%\ref{th}, it can be shown that for the defining ideal of a quasi
%star configuration $Z_d$, one has
%\begin{align*}
%(d+1)/2 \le \widehat{\alpha}(I) \le (d + \sqrt{d})/2.
%\end{align*}
%\end{rem}
\section{Some Applications }%to the case of quasi star configuration}

In this section we give some applications of Theorem \ref{main th}.
As a first consequence of this theorem, we have:
\begin{cor}\label{cor2}
Let $ 0< \varepsilon < \frac{1}{2}$ be a real number.
Then there exists a radical ideal of points $I_{\varepsilon}$ in
$\mathbb{K}[\mathbb{P}^2]$ such that $\rho(I_{\varepsilon}) \in [2- \varepsilon ,2).$
\end{cor}
\begin{proof}
Let $d \geq (\frac{2}{\varepsilon}-1)^2$ be an integer and
let $Z_d$ be the quasi star configuration associated to $d$.
Since $\varepsilon <\frac{1}{2}$ and since
$d \geq (\frac{2}{\varepsilon}-1)^2$, one can see that $d \geq 10$ and
$2-\varepsilon \leq 2-\frac{2}{\sqrt{d}+1}$.
Now by Theorem \ref{main th}, we have
\begin{center}
$2-\varepsilon \leq 2-\frac{2}{\sqrt{d}+1} \leq \rho(I(Z_d))
\leq 2- \frac{2}{d+1} < 2$.
\end{center}
\end{proof}

Known radical ideals of points $I$ in $\mathbb{K}[\mathbb{P}^2]$ for which the containment
$I^{(3)} \subseteq I^2$ fails, are rare,
and even it is not known for which positive integers $r\ge 3$ the containment $I^{(2r-1)} \subseteq I^r$ fails.
As a corollary of Theorem \ref{main th}, we can construct ideals
of points such that they may be a candidate for the
 failure of
$I^{(2r-1)} \subseteq I^r$.
\begin{cor}\label{cor1}
Let $r \geq 2$ be an integer. Then there exists a radical ideal of points $I$ such that
it has the necessary condition for the failure $I^{(2r-1)} \subseteq I^r$, i.e.,
$\rho(I) \geq \frac{2r-1}{r}$.
\end{cor}
\begin{proof}
Let $d \geq (2r-1)^2$ be an integer and let
$Z_d=\sum_{i=1}^{d}q_i + S_2(2,d)$
be the quasi star configuration of points in
$\mathbb{P}^2$.
Since $d \geq (2r-1)^2$, one can see that
$\frac{2r-1}{r} \leq 2-\frac{2}{\sqrt{d}+1}$.
Now, by Theorem \ref{main th}, we have
$\rho(I) \geq 2-\frac{2}{\sqrt{d}+1} \geq \frac{2r-1}{r}$.
\end{proof}

Harbourne and Bocci in \cite{BH}, showed that for a homogeneous
ideal $I$ of $\mathbb{K}[\mathbb{P}^N]$ the containment $I^{(m)}
\subseteq I^r$ for the bound $\frac{m}{r} \geq N$ is optimal. As a
consequence of Theorem \ref{main th}, when $N=2$, we can show that
the quasi star configurations meet this optimality too.
\begin{cor}\label{cor}
Let $I$ be a homogeneous ideal of $\mathbb{K}[\mathbb{P}^2]$ and let
$m$ and $r$ be two positive integers. Then for $m/r \ge 2$, the
containment $I^{(m)} \subseteq I^r$ is optimal.
\end{cor}
\begin{proof}
On the contrary, let there exists a real number $c <2$ such that for
any two positive integers $m,r$ with $m/r \geq c$ the containment
$I^{(m)} \subseteq I^r$  holds. Then we have $\rho(I) \leq c $. Now
let $Z_d$ be a quasi star configuration of points in $\mathbb{P}^2$
with  defining ideal  $I(Z_d)$. Let $d$ tends to infinity.
%Let $d$ tend to infinity and let $I(Z_d)$, which is not a
%fat points ideal any more, be the defining ideal  $I(Z_d)$.
Then by Theorem \ref{main th},
$2=\rho(I(Z_d)) \leq c < 2$, which is a contradiction.
\end{proof}

\paragraph*{\bf Acknowledgement.} The authors would like to thank Alexandra Seceleanu for her very helpful comments
on the minimal free resolution of the ideal of points.

%%%%%%%%%%%%%%%%%%%%%%%%%%%%%%%%%%%%%%%%%%%%%%%%%%%
%%%%%%%%%%%%%%%%%%%%%%%%%%%%%%%%%%%%%%%%%%%%%%%%%%%

\end{document}